\documentclass[oneside,english]{amsart}
\usepackage[T1]{fontenc}
\usepackage[latin9]{inputenc}
\usepackage{amsthm}
\usepackage{amstext}
\usepackage{amssymb}
\usepackage{esint}
\usepackage{amscd}
\makeatletter
\numberwithin{equation}{section}
\numberwithin{figure}{section}
\theoremstyle{plain}
\newtheorem{thm}{\protect\theoremname}[section]

\theoremstyle{plain}
\theoremstyle{definition}

\theoremstyle{plain}
\newtheorem{lem}[thm]{\protect\lemmaname}
\newtheorem{cor}[thm]{\protect\corollaryname}
\theoremstyle{plain}
\newtheorem{rem}[thm]{\protect\remarkname}
\theoremstyle{plain}
\makeatother

\usepackage{babel}
\providecommand{\definitionname}{Definition}
\providecommand{\lemmaname}{Lemma}
\providecommand{\theoremname}{Theorem}
\providecommand{\corollaryname}{Corollary}
\providecommand{\remarkname}{Remark}
\providecommand{\propositionname}{Proposition}

\DeclareMathOperator{\loc}{loc}

\DeclareMathOperator{\ess}{ess}

\DeclareMathOperator{\cp}{cap}

\usepackage{color}

\begin{document}

\title[Capacity inequalities and Lipschitz continuity]{Capacity inequalities and Lipschitz continuity of mappings}

\author{R.~Salimov, E.~Sevost'yanov, A.~Ukhlov}
\begin{abstract}
In this paper we consider topological mappings defined by $p$-capacity inequalities in domains of $\mathbb R^n$. In the case $p=n-1$ we prove the Lipschitz continuity of such mappings, that extends the result by F.~W.~Gehring.
\end{abstract}
\maketitle

\section{Introduction }

This article is devoted to the study of mappings defined by capacity (moduli) inequalities, which have been actively studied in the recent years, see, for example, \cite{Cr16}--\cite{Cr21}, \cite{Ger69}, \cite{GSS15},
\cite{MRSY09}, \cite{Sev11} and \cite{SS11}. In this article we consider homeomorphic mappings $\varphi: \Omega\to\widetilde{\Omega}$, where $\Omega,
\widetilde{\Omega}$  are domains in $\mathbb R^n$, defined by $p$-capacity inequalities
\begin{equation}
\label{def_pp}
\cp_p(\varphi(F_0),\varphi(F_1);\widetilde{\Omega})\leq K_p^p \cp_p(F_0,F_1;\Omega), \,\,1<p<\infty.
\end{equation}
In the case $p=n$ we have usual quasiconformal mappings \cite{V71} and in the case $p\ne n$
this class of mappings was introduced in \cite{Ger69}. In accordance with \cite{VU04,VU05} we define a homeomorphic mapping $\varphi: \Omega\to\widetilde{\Omega}$ as the mapping of bounded $p$-capacitory distortion, if the inequality (\ref{def_pp}) holds for any condenser $(F_0,F_1)\subset\Omega$.

The first topic of the article is devoted to the characterization of homeomorphic mappings defined by $p$-capacity inequalities (\ref{def_pp}) in terms of the inner $p$-dilatation. In the case of mappings with the conformal moduli inequalities of the Poletsky type, the estimates of the inner dilatation was obtained in \cite{SS11A}. Similar estimates of dilatation  in the case of the $p$-modulus, $n-1<p\leqslant n,$ were obtained in~\cite{GS12} and~\cite{GSS17} for homeomorphisms and mappings with a branching, respectively. In this article, we prove:

\vskip 0.2cm
\noindent
{\it Let  $\varphi: \Omega \to \widetilde{\Omega}$ be a homeomorphic mapping.
Then $\varphi$ is the mapping of bounded $p$-capacitory distortion, $p>n-1$, if and only if $\varphi\in W^1_{p',\loc}(\Omega)$, $p'=p/(p-n+1)$, has finite distortion and
$$
\ess\sup\limits_{x\in\Omega}\left(\frac{|J(x,\varphi)|}{l(D\varphi(x))^p}\right)^{\frac{1}{p}}=K_p<\infty, \,\,p>n-1.
$$}
\vskip 0.2cm

The second  topic of the article is devoted the continuity of mappings in the sense of Lipschitz. In \cite{Ger69} the Lipschitz continuity of the mapping of bounded $p$-capacitory distortion proved in the case $n-1<p<n$ and an example was given, that in the case $1<p<n-1$ the Lipschitz continuity does not hold.  Results on of such type have been obtained for of mappings of finite distortion with some restrictions, see, e.g., \cite{GG09}, \cite{MRSY09} and \cite{RSS20}. In the present work by using the methods of the composition operators theory \cite{VU04,VU05} we study analytical properties of these mappings and we prove the Lipschitz continuity in the limit case $p=n-1$.

\vskip 0.2cm
\noindent
{\it Let  $\varphi: \Omega \to \widetilde{\Omega}$ be a homeomorphic mapping of bounded $(n-1)$-capacitory distortion.
Then $\varphi$ belongs to the Sobolev space $L^1_{\infty}(\Omega)$.}
\vskip 0.2cm

This result extends the result by F.~W.~Gehring \cite{Ger69} to the limit case $p=n-1$.

\section{Composition operators and capacity inequalities}

\subsection{Sobolev spaces and composition operators}

Let $\Omega$ be an open subset of $\mathbb R^n$, $n\geq 2$, the Sobolev space $W^1_p(\Omega)$, $1\leq p\leq\infty$, is defined
as a Banach space of locally integrable weakly differentiable functions
$f:\Omega\to\mathbb{R}$ equipped with the following norm:
\[
\|f\mid W^1_p(\Omega)\|=\| f\mid L_p(\Omega)\|+\|\nabla f\mid L_p(\Omega)\|.
\]
The Sobolev space $W^{1}_{p,\loc}(\Omega)$ is defined as a space of functions $f\in W^{1}_{p}(U)$ for every open and bounded set $U\subset  \Omega$ such that
$\overline{U}  \subset \Omega$.

The homogeneous seminormed Sobolev space $L^1_p(\Omega)$, $1\leq p\leq\infty$, is defined as a space
of locally integrable weakly differentiable functions $f:\Omega\to\mathbb{R}$ equipped
with the following seminorm:
\[
\|f\mid L^1_p(\Omega)\|=\|\nabla f\mid L_p(\Omega)\|.
\]
Recall that in Lipschitz domains $\Omega\subset\mathbb R^n$, $n\geq 2$, Sobolev spaces $W^1_p(\Omega)$ and $L^1_p(\Omega)$ are coincide (see, for example, \cite{M}).

In accordance with the non-linear capacity theory \cite{MH72} we consider elements of Sobolev spaces $W^{1,p}(\Omega)$ as  classes of equivalence up to a set of $p$-capacity zero  \cite{M}. 

Let $\Omega$ and $\widetilde{\Omega}$ be domains in the Euclidean space $\mathbb R^n$. Then a homeomorphism $\varphi:\Omega\to\widetilde{\Omega}$ belongs to the Sobolev space $W^1_{p,\loc}(\Omega)$ ($L^1_p(\Omega)$), if its coordinate functions belong to $W^1_{p,\loc}(\Omega)$ ($L^1_p(\Omega)$). In this case, the formal Jacobi matrix $D\varphi(x)$ and its determinant (Jacobian) $J(x,\varphi)$ are well defined at almost all points $x\in\Omega$. We denote
$$
|D\varphi(x)|:=\max\limits_{|v|=1}|D\varphi(x)\cdot v|\,\,\text{and}\,\,l(D\varphi(x)):=\min\limits_{|v|=1}|D\varphi(x)\cdot v|
$$
the maximal dilatation of the linear operator $D\varphi(x)$ and the minimal dilatation of the linear operator $D\varphi(x)$ correspondly.

Let $\Omega$ and $\widetilde{\Omega}$ be domains in $\mathbb R^n$, $n\geq 2$. Then
a homeomorphic mapping $\varphi:\Omega\to\widetilde{\Omega}$ induces a bounded composition
operator \cite{VU04,VU05}
\[
\varphi^{\ast}:L^1_p(\widetilde{\Omega})\to L^1_q(\Omega),\,\,\,1\leq q\leq p\leq\infty,
\]
by the composition rule $\varphi^{\ast}(f)=f\circ\varphi$, if for
any function $f\in L^1_p(\widetilde{\Omega})$, the composition $\varphi^{\ast}(f)\in L^1_q(\Omega)$
is defined quasi-everywhere in $\Omega$ and there exists a constant $K_{p,q}(\Omega)<\infty$ such that
\[
\|\varphi^{\ast}(f)\mid L^1_q(\Omega)\|\leq K_{p,q}(\Omega)\|f\mid L^1_p(\widetilde{\Omega})\|.
\]

The problem of the characterization of mappings generate bounded composition operators on Sobolev spaces arises to the Reshetnyak Problem (1968) and is closely connected with the quasiconformal mappings theory \cite{VG75}. The solution of this problem   is given by the following theorem \cite{U93} (see, also \cite{VU04,VU05} and \cite{GU10} for the case $p=\infty$).

Recall that a $p$-distortion of a mapping $\varphi$ at a point $x\in\Omega$ is defined as
$$
K_p(x)=\inf \{k(x): |D\varphi(x)|\leq k(x) |J(x,\varphi)|^{\frac{1}{p}},\,\,x\in\Omega \}.
$$
In the case $p=n$ we have the usual conformal dilatation and in the case $p\ne n$ the $p$-dilatation arises in \cite{Ger69} (see, also, \cite{V88}).

\begin{thm}
\label{CompTh} Let $\varphi:\Omega\to\widetilde{\Omega}$ be a homeomorphic mapping between two domains $\Omega$ and $\widetilde{\Omega}$. Then $\varphi$ generates a bounded composition operator
\[
\varphi^{\ast}:L^1_p(\widetilde{\Omega})\to L^1_q(\Omega),\,\,\,1\leq q\leq p\leq\infty,
\]
 if and only if $\varphi$ is a Sobolev mapping of the class $W^1_{q,\loc}(\Omega;\widetilde{\Omega})$, has finite distortion
and
\[
K_{p,q}(\varphi;\Omega)=\|K_p \mid L_{\kappa}(\Omega)\|<\infty,
\]
where $1/q-1/p=1/{\kappa}$ ($\kappa=\infty$, if $p=q$).
\end{thm}

The following theorem give properties of mappings, which are inverse to mappings generate bounded composition operators on Sobolev spaces.

\begin{thm}
\label{CompThD} Let a homeomorphic mapping $\varphi:\Omega\to\widetilde{\Omega}$
between two domains $\Omega$ and $\widetilde{\Omega}$ generate a bounded composition
operator
\[
\varphi^{\ast}:L^1_p(\widetilde{\Omega})\to L^1_{q}(\Omega),\,\,\,n-1<q \leq p< \infty.
\]
Then the inverse mapping $\varphi^{-1}:\widetilde{\Omega}\to\Omega$ generates a bounded composition operator
\[
\left(\varphi^{-1}\right)^{\ast}:L^1_{q'}(\Omega)\to L^1_{p'}(\widetilde{\Omega}),
\]
where $p'=p/(p-n+1)$, $q'=q/(q-n+1)$.
\end{thm}

\subsection{Capacity inequalities}
Recall the notion of the variational $p$-capacity \cite{GResh}. The condenser in the domain $\Omega\subset \mathbb R^n$ is the pair $(F_0,F_1)$ of connected closed relatively to $\Omega$ sets $F_0,F_1\subset \Omega$. A continuous function $u\in L_p^1(\Omega)$ is called an admissible function for the condenser $(F_0,F_1)$,
if the set $F_i\cap \Omega$ is contained in some connected component of the set $\operatorname{Int}\{x\vert u(x)=i\}$,\ $i=0,1$. We call $p$-capacity of the condenser $(F_0,F_1)$ relatively to domain $\Omega$
the value
$$
{{\cp}}_p(F_0,F_1;\Omega)=\inf\|u\vert L_p^1(\Omega)\|^p,
$$
where the greatest lower bond is taken over all admissible for the condenser $(F_0,F_1)\subset\Omega$ functions. If the condenser have no admissible functions we put the capacity is equal to infinity.

Let $\varphi:\Omega\to\widetilde{\Omega}$ be a homeomorphic mapping between two domains $\Omega$ and $\widetilde{\Omega}$. Then $\varphi$ is called the mapping of bounded $p$-capacitory distortion, if the inequality
\begin{equation}
\label{CE1}
\cp_p(\varphi(F_0),\varphi(F_1);\widetilde{\Omega})\leq K_p^p \cp_p(F_0,F_1;\Omega), \,\,1<p<\infty.
\end{equation}
holds for any condenser $(F_0,F_1)\subset\Omega$.

\begin{thm}\label{pp}
Let  $\varphi: \Omega \to \widetilde{\Omega}$ be a homeomorphic mapping.
Then $\varphi$ is the mapping of bounded $p$-capacitory distortion, $p>n-1$, if and only if $\varphi\in W^1_{p',\loc}(\Omega)$, $p'=p/(p-n+1)$, has finite distortion and
$$
\ess\sup\limits_{x\in\Omega}\left(\frac{|J(x,\varphi)|}{l(D\varphi(x))^p}\right)^{\frac{1}{p}}=K_p<\infty, \,\,p>n-1.
$$
\end{thm}

\begin{proof}
Consider the inverse mapping $\psi:=\varphi^{-1}:\widetilde{\Omega}\to  \Omega$. The the inequality~(\ref{CE1}) is equivalent to the inequality
\begin{equation}
\label{CE2}
\cp_{p}(\psi^{-1}(F_0),\psi^{-1}(F_1);\widetilde{\Omega})\leq K_{p}^{p} \cp_{p}(F_0,F_1;\Omega).
\end{equation}
So, by \cite{U93,VU98}, the inverse mapping $\varphi^{-1}$ generates a bounded composition operator
$$
\left(\varphi^{-1}\right)^{\ast}:L^1_{p}(\Omega)\to L^1_{p}(\widetilde{\Omega})
$$
and is a $p$-quasiconformal mapping $\varphi^{-1}:\widetilde{\Omega}\to  \Omega$  \cite{GGR95,U93}. Hence the  mapping $\varphi^{-1}$ has the following properties \cite{U93,VU02}:

\medskip
\noindent
1. The mapping $\varphi^{-1}\in W^1_{p,\loc}(\widetilde{\Omega})$, has finite distortion and
$$
\left(\frac{|D\varphi^{-1}(y)|^{p}}{|J(y,\varphi^{-1})|}\right)^{\frac{1}{p}}\leq K_p\,\,\text{for almost all}\,\,y\in \widetilde{\Omega}.
$$

\medskip
\noindent
2. The mapping $\varphi^{-1}$ is differentiable a.e. in $\widetilde{\Omega}$.

\medskip
\noindent
3. The mapping $\varphi^{-1}$ possesses the Luzin $N^{-1}$-property if $n-1<p<n$ ($\varphi$ possesses the Luzin $N$-property).

\medskip
\noindent
4. The mapping $\varphi^{-1}$ possesses the Luzin $N$-property if $n<p<\infty$ ($\varphi$ possesses the Luzin $N^{-1}$-property).

\medskip
\noindent
5. The mapping $\varphi^{-1}$ possesses the Luzin $N$-property and the Luzin $N^{-1}$-property if $p=n$ \cite{V71}($\varphi$ possesses the Luzin $N^{-1}$-property).
\medskip

Now by Theorem~\ref{CompThD} the mapping $\varphi$ generates a bounded composition operator
$$
\varphi^{\ast}: L^1_{p'}(\widetilde{\Omega})\to L^1_{p'}(\Omega),\,\, p'={p}/{(p-n+1)}.
$$
Hence the mapping $\varphi\in W^1_{p',\loc}(\Omega)$, has finite distortion and is differentiable a.e. in $\Omega$ \cite{U93,VU98}.

Denote by $\widetilde{Z}=\{y\in\widetilde{\Omega}: J(y,\varphi^{-1}) =0\}$.
The set $\widetilde{S}\subset \widetilde{\Omega}$, $|\widetilde{S}|=0$, is the set such that on set $\widetilde{\Omega}\setminus \widetilde{S}$ the mapping $\varphi^{-1}:\widetilde{\Omega}\to  \Omega$ has the Luzin $N$-property \cite{H93}.

Then by the change of variables formula \cite{F69,H93} $|\varphi^{-1}(\widetilde{Z}\setminus \widetilde{S})|=0$ and on the set
$\varphi^{-1}(\widetilde{S}\setminus\widetilde{Z})$ we have $J(x,\varphi)=0$ for almost all $x\in \varphi^{-1}(\widetilde{S}\setminus\widetilde{Z})$. Hence, for almost all $x\in \Omega\setminus \varphi^{-1}(\widetilde{Z}\cup\widetilde{S})$ we have:
$$
|J(x,\varphi)|=|J(y,\varphi^{-1})|^{-1}, \,\,y=\varphi(x),
$$
and
$$
l(D\varphi(x))=|D\varphi^{-1}(y)|^{-1}, \,\,y=\varphi(x).
$$
Hence, setting
$$
\left(\frac{|J(x,\varphi)|}{l(D\varphi(x))^p}\right)^{\frac{1}{p}}=0
$$
on the set $Z=\{x\in\Omega: J(x,\varphi)\}$ we obtain
$$
\ess\sup\limits_{x\in\Omega}\left(\frac{|J(x,\varphi)|}{l(D\varphi(x))^p}\right)^{\frac{1}{p}}=
\ess\sup\limits_{y\in\widetilde{\Omega}}\left(\frac{|D\varphi^{-1}(\varphi^{-1}(y))|^{p}}{|J((\varphi^{-1}(y)),\varphi^{-1})|}\right)^{\frac{1}{p}}\leq
K_p<\infty.
$$

\end{proof}

\begin{rem}
\label{rem1}
The assertion of Theorem~\ref{pp} is correct in the case $1\leq p\leq n-1$ with additional assumptions that $\varphi\in W^1_{1,\loc}(\Omega)$ and $\varphi$ is differentiable a.e. in $\Omega$
\end{rem}

\section{On the  Lipschitz continuity of mapping of bounded $p$-capacitory distortion}
Now we consider the Lipschitz continuity of homeomorphic mappings of bounded $p$-capacitory distortion in the case $p=n-1$.

Let $(F_0,F_1)$  be a condenser in the domain $\Omega\subset \mathbb R^n$ such that ${{\cp}}_p(F_0,F_1;\Omega)<\infty$.
Suppose that a function $v$ belongs to $L_p^1(\Omega)$ is admissible for the condenser $(F_0,F_1)$. Then $v$ is called {\it an extremal function for the condenser $(F_0,F_1)$} \cite{VG77}, if
$$
\int\limits_{\Omega\backslash(F_0\cup F_1)}|\nabla v|^p\,dx=
{{\cp}}_p(F_0,F_1;\Omega).
$$
Note that for any $1<p<\infty$ and any condenser $(F_0,F_1)$ with ${{\cp}}_p(F_0,F_1;\Omega)<\infty$ the extremal function exists and unique.

The set of extremal functions for $p$-capacity of every possible pairs of $n$-dimensional connected compacts 
$F_0,\;F_1\subset \Omega$, having smooth boundaries, we denote by the symbol $E_p(\Omega)$. Then the following approximation holds.

\begin{thm}
\label{thm:ApproxExrem}
\cite{VG77}
Let $1<p<\infty$. Then there exists a countable collection of functions $v_i\in E_p(\Omega)$, $i\in \mathbb N$,
such that for every function $u\in L_{p}^{1}(\Omega)$ and for any $\varepsilon>0$
there exists a presentation of $u$ in the form
$u=c_0+\sum\limits_{i=1}^{\infty}c_iv_i$,
for which the inequalities 
$$
\|u\mid L_{p}^{1}(\Omega)\|^p\leq \sum\limits_{i=1}^{\infty}
\|c_iv_i\mid L_{p}^{1}(\Omega)\|^p\leq
        \|u\mid L_{p}^{1}(\Omega)\|^p+\varepsilon
$$ 
hold.
\end{thm}

The following theorem was not formulated, but proved in \cite{U93} by using the approximation by extremal functions (see, also, \cite{VU98}).

\begin{thm}
\label{theorem:CapacityDescPP}
Let $1<p<\infty$.
A homeomorphism $\varphi :\Omega\to \widetilde{\Omega}$
generates a bounded composition operator
$$
\varphi^{\ast}: L^1_p(\widetilde{\Omega})\to L^1_p(\Omega)
$$
if and only if for every condenser 
$(F_0,F_1)\subset\widetilde{\Omega}$
the inequality
$$
\cp_{p}^{\frac{1}{p}}(\varphi^{-1}(F_0),\varphi^{-1}(F_1);\Omega)
\leq K_{p,p}(\varphi;\Omega)\cp_{p}^{\frac{1}{p}}(F_0,F_1;\widetilde{\Omega})
$$
holds. 
\end{thm}

Now by using the capacitory characterization of composition operators on Sobolev spaces we prove the Lipschitz continuity of homeomorphic mappings of bounded $(n-1)$-capacitory distortion, that extends the result by F.~W.~Gehring \cite{Ger69}.

\begin{thm}\label{ppl}
Let  $\varphi: \Omega \to \widetilde{\Omega}$ be a homeomorphic mapping of bounded $(n-1)$-capacitory distortion.
Then $\varphi$ belongs to the Sobolev space $L^1_{\infty}(\Omega)$.
\end{thm}

\begin{proof}
Let  $\varphi: \Omega \to \widetilde{\Omega}$ be a homeomorphic mapping of bounded $(n-1)$-capacitory distortion.
Consider the inverse mapping $\psi:=\varphi^{-1}:\widetilde{\Omega}\to  \Omega$. Then the inequality~(\ref{CE1}) is equivalent to the inequality
\begin{equation}
\label{CE3}
\cp_{n-1}(\psi^{-1}(F_0),\psi^{-1}(F_1);\widetilde{\Omega})\leq K_{n-1}^{n-1} \cp_{n-1}(F_0,F_1;\Omega).
\end{equation}

So, by Theorem~\ref{theorem:CapacityDescPP}, the inverse mapping $\varphi^{-1}$ generates a bounded composition operator
$$
\left(\varphi^{-1}\right)^{\ast}:L^1_{n-1}(\Omega)\to L^1_{n-1}(\widetilde{\Omega})
$$
and is a $(n-1)$-quasiconformal mapping $\varphi^{-1}:\widetilde{\Omega}\to  \Omega$  \cite{GGR95,U93}:
$$
\ess\sup\limits_{y\in \widetilde{\Omega}}\left(\frac{|D\varphi^{-1}(y)|^{n-1}}{J(y,\varphi^{-1})|}\right)^{\frac{1}{n-1}}=K_{n-1}(\varphi^{-1};\widetilde{\Omega})<\infty.
$$

Hence, by \cite{GU17}, the mapping $\varphi: \Omega \to \widetilde{\Omega}$ generates a bounded composition operator
$$
\varphi^{\ast}: L^1_{\infty}(\widetilde{\Omega})\to L^1_{\infty}(\Omega),
$$
and the inequality
\begin{equation}
\label{test}
\|\varphi^{\ast}(f)\mid L^1_{\infty}(\Omega)\|\leq K_{n-1}^{n-1} \|f\mid L^1_{\infty}(\widetilde{\Omega)}\|
\end{equation}
holds for any function $f\mid L^1_{\infty}(\widetilde{\Omega)}$.

Now substituting in the inequality~(\ref{test}) the test functions $f=y_i$, $y=1,...,n$, 
where $y_i$ is the $i$-coordinate of $y\in\widetilde{\Omega}$, we have  $\varphi_i\in L^1_{\infty}(\Omega)$, $i=1,...,n$, and 
$$
\|\varphi_i\mid L^1_{\infty}(\Omega)\|\leq K_{n-1}^{n-1}, \,\,i=1,...,n.
$$

Hence, the mapping $\varphi: \Omega \to \widetilde{\Omega}$ belongs to the Sobolev space $L^1_{\infty}(\Omega)$.

\end{proof}

\section{On differentiability of mapping of bounded $p$-capacitory distortion}

By using Theorem~\ref{ppl} and \cite[Theorem~3]{Ger69} we have the following assertion 

\begin{thm}
\label{lip}
Let  $\varphi: \Omega \to \widetilde{\Omega}$ be a homeomorphic mapping of bounded $p$-capacitory distortion, $n-1\leq p< n$.
Then $\varphi$ is a locally Lipschitz mappings and is differentiable a.e. in $\Omega$.
\end{thm}

In the work \cite{Ger69} it was proved the following estimate of Jacobians of mapping of bounded $p$-capacitory distortion.

\begin{lem}\label{Jacp}
Let  $\varphi: \Omega \to \widetilde{\Omega}$ be a homeomorphic mapping of bounded $p$-capacitory distortion, $1\leq p<n$.
Then $|J(x, \varphi)|\leq (K^p_p)^{n/(n-p)}$.
\end{lem}

By using Lemma~\ref{Jacp} we obtain estimates of the Lipschitz constants of homeomorphic mapping of bounded $p$-capacitory distortion.

\begin{cor}
Let  $\varphi: \Omega \to \widetilde{\Omega}$ be a homeomorphic mapping of bounded $p$-capacitory distortion, $n-1\leq p<n$.
Then $|D\varphi(x)|\leq (K^p_{p})^{1/(n-p)}$.
\end{cor}

\begin{proof} Indeed, it is easy to check that
$$|D\varphi(x)|^p\leq |J(x,\varphi)|^{p-n+1} \left(\frac{|J(x,\varphi)|}{l(D\varphi(x))^p}\right)^{n-1}\,.$$
By Theorem \ref{pp},  Corollary \ref{Jacp} and Remark \ref{rem1},  we have
$$|D\varphi(x)|^p\leq |J(x,\varphi)|^{p-n+1} \left(\frac{|J(x,\varphi)|}{l(D\varphi(x))^p}\right)^{n-1}\leq \left(K_p^{n/(n-p)}\right)^{p-n+1}\,.$$
\end{proof}

Now, by using Theorem~\ref{lip} we have

\begin{cor}
Let  $\varphi: \Omega \to \widetilde{\Omega}$ be a homeomorphic mapping of bounded $p$-capacitory distortion, $n-1\leq p<n$.
Then $$\limsup\limits_{x\to x_0} \frac{|f(x)-f(x_0)|}{|x-x_0|}  \leq (K^p_{p})^{1/(n-p)} $$ 
for almost all $x\in \Omega$.
\end{cor}

\vskip 0.5cm

\noindent
Ruslan Salimov; Institute of Mathematics of NAS of Ukraine, Tereschenkivs'ka Str. 3, 01 601 Kyjiv, Ukraine

\noindent
\emph{E-mail address:} \email{ruslan.salimov1@gmail.com} \\

\noindent
Evgeny Sevost'yanov; Department of Mathematical Analysis, Zhytomyr Ivan Franko State University, 40 Bol'shaya Berdichevskaya Str., Zhytomyr, 10008, Ukraine

\noindent
Institute of Applied Mathematics and Mechanics of NAS of Ukraine, 1 Dobrovol'skogo Str., 84 100 Slavyansk,  Ukraine

\noindent
\emph{E-mail address:} \email{esevostyanov2009@gmail.com} \\

\noindent
Alexander Ukhlov; Department of Mathematics, Ben-Gurion University of the Negev, P.O.Box 653, Beer Sheva, 8410501, Israel

\noindent
\emph{E-mail address:} \email{ukhlov@math.bgu.ac.il

\end{document}